\documentclass[11pt,letterpaper]{amsart}
\usepackage{calrsfs}

\usepackage{graphicx}
\usepackage{hyperref}
\usepackage{amsthm}
\usepackage{amssymb}
\usepackage{multirow}
\usepackage{tikz-cd}
\usepackage{amsmath}
\usepackage{todonotes}
\usepackage{amsbsy}
\usepackage[all]{xy}
\usepackage{qtree}

\usepackage{enumitem}
\usepackage{xfrac}    
\usepackage{color, colortbl}
\definecolor{LightCyan}{rgb}{0.88,1,1}
\definecolor{Gray}{gray}{0.9}

\usepackage{faktor}

\usepackage{changes}
\definechangesauthor[color=orange,name={Aristides Kontogeorgis}]{AK}
\definechangesauthor[color=blue,name={Alex Terezakis}]{AT}

\newtheorem{theorem}{Theorem}
\newtheorem{lemma}[theorem]{Lemma}
\newtheorem{corollary}[theorem]{Corollary}
\newtheorem{proposition}[theorem]{Proposition}
\theoremstyle{definition}

\newtheorem{criterion}[theorem]{Criterion}
\newtheorem{remark}[theorem]{Remark}
\newtheorem{definition}[theorem]{Definition}

\DeclareMathOperator\Aut{Aut}
\DeclareMathOperator\Gl {GL}
\DeclareMathOperator\Sym{Sym}

\DeclareMathOperator{\Spe}{Spec}

\newcommand{\lf}{\left\lfloor}
\newcommand{\rf}{\right\rfloor}

\newcommand{\Z}{\mathbb{Z}}

\renewcommand{\O}{{\mathcal{O}}}
\renewcommand{\mod}{\bmod}

\date{\today}

\title{An obstruction to the local lifting problem}

\author[A. Kontogeorgis]{Aristides Kontogeorgis}
\address{Department of Mathematics, National and Kapodistrian  University of Athens
Pane\-pist\-imioupolis, 15784 Athens, Greece}
\email{kontogar@math.uoa.gr}

\author[A. Terezakis]{Alexios Terezakis }
\address{Department of Mathematics, National and Kapodistrian University of Athens\\
Panepistimioupolis, 15784 Athens, Greece}
\email{aleksistere@math.uoa.gr}

\date \today

\makeatletter
\newcommand{\aprod}{\mathop{\operator@font \hbox{\Large$\ast$}}}
\makeatother

\begin{document}

\keywords{Branched cover, Lifting of representations, KGB-obstruction  Generalized Oort conjecture, Galois group, metacyclic group, dihedral group}

\subjclass{14H37, 12F10; 11G20, 13B05, 13F35, 14H30}

\begin{abstract}
We are investigating the 
lifting problem for local actions
involving semidirect products of a cyclic $p$-group with a cyclic group prime to $p$, where $p$ represents the characteristic of the special fiber. We establish a criterion based on the Harbater-Katz-Gabber compactification of local actions, enabling us to determine whether a given local action can be lifted or not. Specifically, in the case of the dihedral group, we present an example of a local dihedral action that cannot be lifted. This instance provides a more potent obstruction than the KGB obstruction.

\bigskip

 Nous étudions le problème de relèvement des actions locales des produits semi-directs d'un groupe cyclique $p$ par un groupe cyclique d'ordre premier avec $p$, où $p$ est la caractéristique de la fibre spéciale. Nous  obtenons un critère basé sur la compactification des actions locales de Harbater-Katz-Gabber, qui nous permet de décider si une action locale peut être relevée ou non. En particulier, dans le cas du groupe diédral, nous donnons un exemple d'action locale diédrale qui ne peut pas être relevée, offrant ainsi une obstruction plus forte que l'obstruction KGB.
 \end{abstract}
\maketitle


\section{Introduction}

Let $G$ be a finite group, $k$ an algebraically closed field of characteristic $p>0$ and consider the homomorphism  
\[
\rho: G \hookrightarrow \Aut(k[[t]]), 
\]
which will be called {\em a local $G$-action}. Let $W(k)$ denote the ring of Witt vectors of $k$. 
The local lifting problem addresses the question:
Does there exist an extension $\Lambda/W(k)$, that is $\Lambda$ is an integrally closed domain contained in a field extension of $\mathrm{Frac}(W(k))$, and a representation 
\[
\tilde{\rho}: G \hookrightarrow \Aut(\Lambda[[T]]),
\]
such that if $t$ is the reduction of $T$, then the action of $G$ on $\Lambda[[T]]$
reduces to the action of $G$ on $k[[t]]$?
If the answer to the above question is affirmative, then we say that the $G$-action lifts to characteristic zero. 
A group $G$ for which every local $G$-action on $k[[t]]$ lifts to characteristic zero is called {\em a local Oort group for $k$}.

Following an examination of specific obstructions, such as the Bertin obstruction, the KGB obstruction, and the Hurwitz tree obstruction, it has been established that the potential candidates for local Oort groups in characteristic $p$ are limited to the following: 
\begin{enumerate}
\item Cyclic groups;
\item Dihedral groups $D_{p^h}$ of order $2p^h$;
\item The alternating group $A_4$.
\end{enumerate}

The Oort conjecture states that every cyclic group $C_q$ of order $q=p^h$ {is a local Oort group}. This conjecture was recently proven by F. Pop \cite{MR3194816} using the work of A. Obus and S. Wewers \cite{ObusWewers}. A. Obus proved that $A_4$ is a local Oort group in \cite{MR3556796} and this was also known to F. Pop as well as I. Bouw and S. Wewers \cite{MR2254623}. Dihedral groups $D_p$ are known to be local Oort by the work of I. Bouw and S. Wewers for $p$ odd \cite{MR2254623} and by the work of G. Pagot \cite{PagotPhd}. Several cases of dihedral groups $D_{p^h}$ for small $p^h$ have been studied by A. Obus \cite{Obus2017} and H. Dang, S. Das, K. Karagiannis, A. Obus, V. Thatte \cite{MR4450616}, while the $D_4$ was studied by B. Weaver \cite{MR3874854}. For further details on the lifting problem, refer to \cite{MR2441248}, \cite{MR2919977}, \cite{MR3591155}, \cite{Obus12}. 

Perhaps the most significant of the currently known obstructions is the KGB obstruction \cite{MR2919977}. It was conjectured that 
if the $p$-Sylow subgroup of $G$ is cyclic, then this is the sole obstruction for the local lifting problem, see \cite{Obus12}, \cite{Obus2017}. 
In particular, the KGB obstruction for the dihedral group $D_q$ is known to vanish, and the so called ``generalized Oort conjecture'' asserts that the local action of $D_q$ always lifts {for $q$-odd}. 

In this article, we will provide a new obstruction for the lifting problem of a $C_q \rtimes C_m$-action and in particular for the group $D_q$. We would like to emphasize that in contrast to the KGB obstruction, which vanishes for the dihedral groups of order $2p^h$, our obstruction does not.
Using this criterion we provide in section \ref{sec:no-lift} a counterexample to the generalized Oort conjecture, 
by proving 
the HKG-cover corresponding to $D_{125}$, with a selection of lower jumps $9,189,4689$ does not lift.

We use the Harbater-Katz-Gabber compactification (referred to as HKG) as a means to construct complete curves from local actions. This approach equips us with a diverse set of tools stemming from the theory of complete curves, thereby allowing us to convert the local action and its deformations into representations of linear groups that act on the differentials of the HKG-curve. The foundational tools for this purpose are detailed in our article \cite{2101.11084}, where we have compiled various insights into the interrelation between lifting local actions, lifting curves, and lifting linear representations.

To elaborate further, let us delve into the specifics. Consider a local action $\rho: G \rightarrow \Aut k[[t]]$, where the group $G$ is $C_q \rtimes C_m$.
The Harbater-Katz-Gabber compactification theorem asserts that there is a Galois cover
$X\rightarrow \mathbb{P}^1$ ramified wildly and completely only at one point $P$ of $X$ with Galois group 
$G=\mathrm{Gal}(X/\mathbb{P}^1)$ 
and tamely on a different point $P'$ with ramification group $C_m$, so that the action of $G$ on the completed local ring $\O_{X,P}$ coincides with the original 
action of $G$ on $k[[t]]$. Notably, it is established that the local action lifts if and only if the corresponding HKG-cover undergoes lifting as well.

In particular, we have proved that in order to lift a subgroup $G \subset \Aut(X)$, the representation $\rho:G\rightarrow \Gl H^0(X,\Omega_X)$ should be lifted to characteristic zero and also the lifting should be compatible with the deformation of the curve. More precisely, in \cite{2101.11084} we have proved 
the following relative version of Petri's theorem.
\begin{proposition}
\label{red-new-prot}
Let $f_1,\ldots,f_r \in S:=\Sym H^0(X,\Omega_X)=k[\omega_1,\ldots,\omega_g]$ be quadratic polynomials which generate the canonical ideal $I_{X}$ of a curve $X$ defined over an algebraic closed field $k$. Any deformation 
$\mathcal{X}_A$ is given by quadratic polynomials $\tilde{f}_1,\ldots,\tilde{f}_r \in \Sym H^0(\mathcal{X}_A,\Omega_{\mathcal{X}_A/A})=A[W_1,\ldots,W_g]$, which reduce to $f_1,\ldots,f_r$ modulo the maximal ideal 
$\mathfrak{m}_A$ of $A$. 
\end{proposition}
Additionally, we have provided the following liftability criterion:
\begin{theorem}
\label{th:main-lift}
Consider an epimorphism $R\rightarrow k \rightarrow 0$ of local Artin rings.
Let $X$ be a curve which is is canonically embedded in $\mathbb{P}^{g-1}_k$ and the canonical ideal is generated by quadratic polynomials; and acted on by the group $G$. 
The curve $X \rightarrow \Spe(k)$ can be lifted to a family $\mathcal{X} \rightarrow \Spe(R) \in D_{\Gl }(R)$ along with the G-action, if and only if the representation $\rho_k :G \rightarrow \Gl _g(k)=\Gl (H^0(X,\Omega_X))$ lifts to a representation  $\rho_{R}:G \rightarrow \Gl _g(R)= 
\Gl (H^0(\mathcal{X},\Omega_{\mathcal{X}/R}))$  and moreover the lift of the canonical ideal is left invariant by the action of $\rho_R(G)$. 
\end{theorem}

In section \ref{sec:HKG-covers} we prove that the canonical ideal of the HKG-cover is  generated by quadratic polynomials, therefore theorem \ref{th:main-lift} can be applied. In order to decide whether a linear representation of $G=C_q \rtimes C_m$ can be lifted, we will use the following criterion for the lifting of the linear representation, based on the decomposition of a $k[G]$-module into indecomposable summands. We begin  by describing the indecomposable 
$k[G]$-modules for the group $G=C_q \rtimes C_m$:
\begin{proposition}
Suppose that the group  $G=C_q \rtimes C_m$ is represented in terms of generators $\sigma,\tau$ and relations as follows:
\[
G=\langle \sigma,\tau | \tau^q=1, \sigma^m=1, \sigma \tau \sigma^{-1}= \tau^{\alpha}\rangle,
\]
for some  $\alpha\in \mathbb{N},
 1\leq \alpha \leq p^h-1, (\alpha, p)=1$.
Every indecomposable $k[G]$-module has dimension $1\leq \kappa \leq q$ and is of the form 
$
  V_{\alpha}(\lambda,\kappa),
$
where the underlying space of $V_{\alpha}(\lambda,\kappa)$ has the set of elements $\{(\tau-1)^\nu e, \nu=0,\ldots,\kappa-1\}$ as a basis for some $e\in V_\alpha(\lambda,\kappa)$, and the action of $\sigma$ on $e$ is given by $\sigma e= \zeta_m^\lambda e$, for a fixed primitive $m$-th root of unity. 
\end{proposition} 
\begin{proof}
A proof can be found in \cite[sec. 3]{MR4779377}. Notice also that $(\tau-1)^{\kappa}e=0$.
\end{proof}

Notice that in section \ref{sec:HolomorphicDifferentialSpecialFibre} we will give an alternative description of the indecomposable $k[G]$-modules, namely the $U_{\ell,\mu}$ notation, which is compatible with the results of \cite{MR4130074}.
\begin{remark}
In the article \cite{MR4779377} of the authors, the $V_{\alpha}(\lambda,\kappa)$ notation is used. In this article we will need the Galois module structure of the space of homolomorphic differentials of a curve and we will employ the results of 
\cite{MR4130074}, where the $U_{\ell,\mu}$ notation is used. These modules will be defined in section \ref{sec:HolomorphicDifferentialSpecialFibre}, notice that 
$V_{\alpha}(\lambda,\kappa)=U_{(\lambda+a_0(\kappa-1)) \mod m,\kappa}$, where $\alpha = \zeta_m^{a_0}$, see lemma \ref{comparisonIndec}.
\end{remark}



\begin{theorem}
\label{cond-lift}
Consider a $k[G]$-module $M$ which is decomposed as a direct sum 
\[
  M= V_{\alpha}(\varepsilon_1,\kappa_1) \oplus \cdots \oplus V_{\alpha}(\varepsilon_s,\kappa_s). 
\]
The module lifts to an $R[G]$-module if and only if 
the set $\{1,\ldots,s\}$ can be written as a disjoint union of sets $I_\nu$, $1\leq \nu \leq t$ so that

\begin{enumerate}[label=\alph*.]
\item \label{cond-lifta}$\sum_{\mu \in I_\nu} \kappa_\mu \leq q$, for all $1\leq \nu \leq t$. 
\item \label{cond-liftb}$\sum_{\mu \in I_\nu} \kappa_\mu\equiv a \mod m$  for all $1\leq \nu \leq t$, where $a\in \{0,1\}$. 
\item \label{cond-liftc} 
For each $\nu$, $1\leq \nu \leq t$ there is an enumeration $\sigma: \{1,\ldots,\#I_\nu\} \rightarrow I_\nu\subset \{1,..,s\}$,
such that 
\[
  \varepsilon_{\sigma(2)}= \varepsilon_{\sigma(1)} \alpha^{\kappa_{\sigma(1)}},
  \varepsilon_{\sigma(3)}= \varepsilon_{\sigma(2)} \alpha^{\kappa_{\sigma(2)}},\ldots,
  \varepsilon_{\sigma(s)}= \varepsilon_{\sigma(s-1)} \alpha^{\kappa_{\sigma(s-1)}}.
\]
\end{enumerate}
Condition \ref{cond-liftb}, with $a=1$ happens only if the lifted $C_q$-action in the generic fibre has an eigenvalue equal to $1$ for the generator $\tau$ of $C_q$.   
\end{theorem}

\begin{proof}
 See \cite{MR4779377}. 
\end{proof}

The idea of the above theorem is that  indecomposable $k[G]$-modules in the decomposition of $H^0(X,\Omega_X)$ of the special fibre, should be combined together in order to give indecomposable modules in the decomposition of holomorphic differentials of the  relative curve. 
Keep in mind that the lifting of the cyclic group acting on a curve of characteristic zero in the generic fibre, has the additional property that every eigenvalue of a generator of $C_q$ is different than one, see {proposition} \ref{prop:no-eigenvalues}. 

Notice that when $m=2$, that is for the case of dihedral groups $D_q$ of order $2q$, there is no need to pair two indecomposable $k[D_q]$-modules together in order to lift them into an indecomposable $R[D_q]$-module. The sets $I_\nu$ can be singletons and the conditions of theorem \ref{cond-lift} are trivially satisfied. For example, condition \ref{cond-lift}.\ref{cond-liftb} does not give any information since every integer is either odd or even.  This means that the linear representations always lift. 

In our geometric setting on the other hand, we know that in the generic fibre cyclic actions do not have identity eigenvalues, see proposition \ref{prop:no-eigenvalues}. This means that we have to consider lifts that satisfy \ref{cond-lift}.\ref{cond-liftb} with $a=0$. Therefore, indecomposable modules for $G=C_q \rtimes C_2 =D_q$ of odd dimension $d_1$ should find another indecomposable module of odd dimension $d_2$ in order to lift to an $R[G]$-indecomposable module of even dimension $d_1+d_2$. Moreover, this dimension should satisfy $d_1+d_2 \leq q$. If we also take care of the condition \ref{cond-lift}.\ref{cond-liftc} we arrive at the following

\begin{criterion}
\label{crit:main}
{If the HKG-curve  acted on by $D_q$ lifts in characteristic zero, then} all indecomposable summands $V_{\alpha}(\varepsilon,d)$, where $\varepsilon\in \{0,1\}$ and $1 \leq d \leq q$ with $d$ odd have a pair 
$V_\alpha (\varepsilon',d')$, with $\varepsilon' \in \{0,1\}- \{\varepsilon\}$ and $d'$ odd and  $d+d'\leq q$. 
Notice that since, $d,d'$ are both odd we have
\[
  V_{\alpha}(\varepsilon,d)=U_{\varepsilon+d-1 \mod 2,d} =
  U_{\varepsilon,d}, \quad V_{\alpha}(\varepsilon',d')=U_{\varepsilon'+d'-1 \mod 2,d'}=U_{\varepsilon',d'}. 
\]
The indecomposable modules given above will be called {\em complementary}. 
We will apply this criterion for complementary modules in the $U_{\varepsilon,d}$-notation. 
\end{criterion}
In order to apply this idea we need a detailed study of the direct $k[G]$-summands of $H^0(X,\Omega_X)$, for $G=C_q \rtimes C_m$. This is considered in section \ref{sec:HolomorphicDifferentialSpecialFibre}, where we employ the joint work of the first author with F. Bleher and T. Chinburg  \cite{MR4130074}, in order to compute the decomposition of $H^0(X,\Omega_X)$ into indecomposable $kG$-modules, in terms of the ramification filtration of the local action. 
Then the lifting criterion of theorem \ref{cond-lift} is applied. 
Our method gives rise to an algorithm which takes as input a group $C_q \rtimes C_m$, with a given sequence of lower jumps and decides whether 
the representation can be lifted or not. This algorithm is implemented in sage 9.8 \cite{sagemath9.8}  and our code is freely available \cite{CodeSage}.

In section \ref{sec:no-lift} 
we give an example of a $C_{125} \rtimes C_4$ HKG-curve which does not lift and then we restrict ourselves to the case of dihedral groups. 
The possible ramification filtrations for local actions of the group $C_q \rtimes C_m$ were computed in the work of A. Obus and R. Pries in \cite{MR2577662}. 
We focus on the case of dihedral groups $D_q$ with lower jumps
\begin{equation}
\label{eq:jumpsPO}
  b_\ell = w_0\frac{p^{2 \ell}+1}{p+1}, 0 \leq \ell \leq h-1.
\end{equation}
For the {value}  $w_0=9$ we will show that the local action of the dihedral group $D_{125}$ does not lift, providing a counterexample to the conjecture that the KGB-obstruction is the only obstruction to the local lifting problem. 



\noindent {\bf Acknowledgements.}
We would like to thank A. Obus for his remarks and comments on an earlier version of this article. 


%
%
%
\section{Notation}

In this article we will study  {\em metacyclic} groups $G=C_q \rtimes C_m$, where $q=p^h$ is a power of the characteristic and $m\in \mathbb{N}, (m,p)=1$. Let $\tau$ be a generator of the cyclic group $C_q$ and $\sigma$ be a generator of the cyclic group $C_m$.  

The group $G$ is given in terms of generators and relations as follows:
\begin{equation}
\label{eq:G-def}
G=\langle \sigma,\tau | \tau^q=1, \sigma^m=1, \sigma \tau \sigma^{-1}= \tau^{\alpha}  \rangle,
\end{equation}
for some $\alpha\in \mathbb{N},
 1\leq \alpha \leq p^h-1, (\alpha, p)=1$.
The integer $\alpha$ satisfies the following congruence:
\begin{equation}
  \label{a-modq}
  \alpha^m \equiv 1 \mod q
\end{equation}
as one sees by computing $\tau=\sigma^m \tau\sigma^{-m}=\tau^{\alpha^m}$. 
Also the integer $\alpha$ can be seen as an element in the finite field $\mathbb{F}_p$, and it is a
 $(p-1)$-th root of unity, not necessarily primitive. In particular the following holds:
\begin{lemma}\label{loga} Let $\zeta_m$ be a fixed primitive $m$-th root of unity. 
There is a natural number  $a_0$, $0\leq a_0 < m-1$ such that $\alpha=\zeta_m^{a_0}$.  
\end{lemma}
\begin{proof}
The integer $\alpha$, if we see it as an element in the field $k$ of characteristic $p>0$, is an element in the finite field $\mathbb{F}_p\subset k$, therefore $\alpha^{p-1}=1$ as an element in $\mathbb{F}_p$.
Let $\mathrm{ord}_p(\alpha)$ be the order of $\alpha$ in $\mathbb{F}_p^*$. By eq. \ref{a-modq} we have that $\mathrm{ord}_p(\alpha) \mid p-1$ and $\mathrm{ord}_p(\alpha) \mid m$, that is $\mathrm{ord}_p(\alpha) \mid (p-1,m)$. 

The primitive $m$-th root of unity $\zeta_m$ generates a finite field $\mathbb{F}_p(\zeta_m)=\mathbb{F}_{p^\nu}$
for some integer $\nu$, which has cyclic multiplicative  group $\mathbb{F}_{p^\nu}\backslash \{0\}$ containing both the cyclic groups $\langle \zeta_m \rangle$ and $\langle \alpha \rangle$. Since  for every divisor $\delta$ of the order of a cyclic group $C$ there is a unique subgroup $C' < C$ of order $\delta$ we have that $\alpha \in \langle \zeta_m \rangle$, and the result follows. 
\end{proof}
\begin{remark}\label{remark:HKB}
  {It is known \cite[prop. 5.9]{Obus12} that}
for the case of $C_q \rtimes C_m$ the KGB-obstruction vanishes if and only if the first lower jump $h$ satisfies $h\equiv -1\mod m$. For this to happen, the conjugation action of $C_m$ on $C_q$ has to be faithful, see \cite[prop. 5.9]{Obus12}. Also notice that by \cite[th. 1.1]{MR2577662}, that if $u_0,u_1,\ldots,u_{h-1}$ is the sequence of upper ramification jumps for the $C_q$ subgroup, then the condition $h\equiv -1 \mod m$ implies that all upper jumps $u_i$ are congruent to $-1$ modulo $m$, that is $u_i \equiv -1 \mod m$. 
\end{remark}

\section{HKG-covers and their canonical ideal}
\label{sec:HKG-covers}
\begin{lemma}
    \label{lemma:ineqHK}
Consider the Harbater-Katz-Gabber curve corresponding to the local group action of $C_q \rtimes C_m$, where $q=p^h$ is a power of the characteristic $p$. If one of the following conditions holds:
\begin{itemize}
\item $h\geq 3$ or $h=2, p>3$ 
\item $h=1$ and the first jump $i_0$ in the ramification filtration for the cyclic group satisfies
$i_0\neq 1$ and $q\geq \frac{12}{i_0-1}+1$,
\end{itemize}
then the curve $X$ has canonical ideal generated by quadratic polynomials.
\end{lemma}
\begin{remark}
  Notice that in \cite{MR4194180} the canonical ideal for HKG-covers is explicitly described.
\end{remark}
\begin{remark}
Notice, that the missing cases in the above lemma which  satisfy the KGB obstruction, are all either cyclic, $D_3$ or $D_9$,  which are all known local Oort groups. 
\end{remark}
\begin{proof}
Using Petri's theorem \cite{Saint-Donat73} we will need to prove that the curve $X$ has genus $g\geq 6$ provided that $p$ or $h$ are as in the statement of lemma \ref{lemma:ineqHK}.
 We will also prove that the curve $X$ is not hyperelliptic nor trigonal.

\begin{remark}
\label{rem:rat-act}
Let us first recall that a cyclic group of order $q=p^h$ for $h\geq2$ cannot act on the rational curve, see \cite[thm 1]{Val-Mad:80}. Also let us recall that a cyclic group of order $p$ can act on a rational curve and in this case the first and only break in the ramification filtration is $i_0=1$. This latter case is excluded.  
\end{remark}

Consider first the case $p^h=p$ and $i_0\neq 1$. 
In this case we compute the genus $g$  of the HKG-curve $X$ using Riemann-Hurwitz formula:
\[
2g=2-2mq + q(m-1) +qm-1+ i_0(q-1), 
\]
where the contribution $q(m-1)$ is from the $q$-points above the unique tame ramified point, while $qm-1+i_0(q-1)$ is the contribution of the wild ramified point. 
This implies that
\[
2g = (i_0-1)(q-1),  
\]
therefore if $i_0\geq2$, it suffices to have $q=p^h \geq 13$ and more generally  it is enough to have $q \geq \frac{12}{i_0-1}+1$ in order to ensure that $g\geq 6$. 

For the case  $h\geq2$, we can write a stronger inequality based on Riemann-Hurwitz theorem as (recall that $i_0\equiv i_1 \mod p$ so $i_0-i_1 \geq p$)
\begin{equation}
\label{ineq-genus}
2g \geq (i_0-1)(p^{h}-1) + (i_0-i_1)(p^{h-1}-1) \geq p^h-p,
\end{equation}
which implies that $g\geq 6$ for $p>3$ or $h\geq 3$.



In order to prove that the curve is not hyperelliptic we observe that the automorphism group of a hyperelliptic curve contains a normal subgroup generated by the hyperelliptic involution $j$, so that $X \rightarrow X/\langle j \rangle=\mathbb{P}^1$. 
It is known that the automorphism group of a hyperelliptic curve fits in the short exact sequence 
\begin{equation}
\label{eq:hyp-aut}
1 \rightarrow \langle j \rangle  \rightarrow \Aut(X) \rightarrow H \rightarrow 1, 
\end{equation}
where $H$ is a subgroup of $\mathrm{PGL}(2,k)$, see
 \cite{BraStich86}. If 
$m$ is odd then the hyperelliptic involution is not an element in $C_m$.  If $m$ is even, let $\sigma$ be a generator of the cyclic group of order $m$ and $\tau$
a generator of the group $C_q$. The involution  $\sigma^{m/2}$ again can't be the hyperelliptic involution. Indeed, the hyperelliptic involution is central, while the conjugation action of $\sigma$ on $\tau$ is faithful that is $\sigma^{m/2} \tau \sigma^{-m/2} \neq \tau$.
In this case $G=C_q \rtimes C_m$ is a subgroup of $H$ which should  act on  the rational function field. By the classification of such groups in \cite[Th. 1]{Val-Mad:80} this is not possible, for $m>2$, while the case $m=2$, i.e. the case of the dihedral group $D_q$ can also be ruled out using Remark \ref{rem:rat-act}.
Thus $X$ can't be hyperelliptic.

We will prove now that the curve is not trigonal. 
Using Clifford's theorem we can show  \cite[B-3 p.137]{MR770932} that a non-hyperelliptic curve of genus $g\geq 5$ cannot have two distinct $g_3^1$. Notice that we have already required the stronger condition $g \geq 6$. 
 So if there is a $g_3^1$, then this is unique. Moreover, 
 the $g_3^1$ gives rise to a map $\pi:X \rightarrow \mathbb{P}^1$ and every automorphism of the curve $X$ fixes this map. Therefore, we obtain 
 a morphism $\phi:C_q \rtimes C_m \rightarrow \mathrm{PGL}_2(k)$ and we 
arrive at the short exact sequence 
\[
1 \rightarrow \mathrm{ker} \phi  \rightarrow C_q \rtimes C_m \rightarrow H \rightarrow 1,
\]
for some finite subgroup $H$ of $\mathrm{PGL}(2,k)$. 
If $\mathrm{ker} \phi=\{1\}$,  then  we have the tower of curves
$X \stackrel{\pi}{\longrightarrow} \mathbb{P}^1 \stackrel{\pi'}{\longrightarrow} \mathbb{P}^1$, where $\pi'$ is a Galois cover with group $C_q \rtimes C_m$. This 
implies that $X$ is a rational curve contradicting 
{Remark}
 \ref{rem:rat-act}. If $\mathrm{ker}\phi$ is a cyclic group of order $3$, then we have that $3\mid m$ and the tower $X \stackrel{\pi}{\longrightarrow} \mathbb{P}^1 \stackrel{\pi'}{\longrightarrow} \mathbb{P}^1$, where $\pi$ is a cyclic Galois cover of order $3$ and $\pi'$ is a Galois cover with group $ C_q \rtimes  C_{m/3}$. As before this contradicts 
{Remark} \ref{rem:rat-act} and is not  possible. 
\end{proof}

\section{Invariant subspaces of vector spaces}
\label{sec:invariant-subspaces}

The $g\times g$ symmetric matrices $A_1,\ldots,A_r$ defining the quadratic canonical ideal of the curve $X$, define a vector subspace of the vector space $V$ of $g\times g$ symmetric matrices. By the Oort conjecture, we know that there is a lifted relative curve $\tilde{X}$ and 
by proposition \ref{red-new-prot} to the lift of $\tilde{X}$ correspond symmetric matrices $\tilde{A}_1,\ldots,\tilde{A}_r$ with entries in a local principal ideal domain $R$, which reduce to the initial matrices $A_1,\ldots,A_r$.
Moreover, by theorem \ref{th:main-lift} the submodule  $\tilde{V}=\langle \tilde{A}_1,\ldots,\tilde{A}_r \rangle$ is left invariant under the action of a lifting $\tilde{\rho}$ of the representation $\rho:C_q \rightarrow \Gl _g(k)$.

\begin{proposition}
\label{prop:no-eigenvalues}
Let $\tilde{g}$ be the genus of the quotient curve $X/H$ for a subgroup $H$ of the automorphism group of a curve $X$ in characteristic zero. We have
\[
\dim H^0(X,\Omega_X^{\otimes d})^{H}=
\begin{cases} \displaystyle
\tilde{g} & \text{ if } d=1 \\
(2d-1)(\tilde{g}-1) + \sum_{P \in X/G} 
\lf d  
\left(
1 - \frac{1}{e(\tilde{P})}
\right) \rf & \text{ if } d> 1
\end{cases}
\]
\end{proposition}
\begin{proof}
See \cite[eq. 2.2.3,2.2.4 p. 254]{Farkas-Kra}.
\end{proof}
Therefore, a generator of $H=C_q$ acting on $H^0(X,\Omega_X)$ has no identity eigenvalues, 
  Thus $m$ should divide $g$. This means that we have to consider liftings of indecomposable summands of  the $C_q$-module $H^0(X,\Omega_X)$, which satisfy condition \ref{cond-lift}.\ref{cond-liftb} with $a=0$.
We now assume that condition \ref{cond-lift}.\ref{cond-liftb} of theorem \ref{cond-lift} can be fulfilled, so there is a lifting
of the representation
\[
\xymatrix{
  & \Gl _g(R) \ar[d]^{\mod \mathfrak{m}_R}\\
  C_q \rtimes C_m \ar[r]^{\rho} 
\ar[ru]^{\tilde{\rho}}
  & \Gl _g(k)
}
\]

{The canonical ideal of the special fibre $X$ corresponds to a vector space $V \subset k^{N}$, $N=\frac{g(g+1)}{2}$, of dimension $r=\binom{g-2}{2}$, see \cite[remark 8]{2101.11084}. This space is acted on by  $G=C_q \rtimes C_m$  in terms of the action given by
\[
\rho^{(1)}(\gamma) (e)= \rho(\gamma )^t e \rho(\gamma), \text{ for } \gamma \in G \text{ and } e \in V.
\]
We select a lifting of this curve together with the action of the cyclic group action $C_q=\langle \tau \rangle$ which exists by the statement of the Oort conjecture for cyclic groups. 
We therefore arrive at a free submodule $\tilde{V} \subset R^N$. If moreover $e_1, \ldots , e_r$ is a basis of $V$ then there are elements  $E_1, \ldots ,E_r$,  which are $g\times g$ matrices with entries in $R$ so that $E_i \equiv e_i \mod \mathfrak{m}_R$ and  forming a free submodule of the module of symmetric $g\times g$ matrices with entries in $R$, together with an action of $C_q$ given by
\[
    \tilde{\rho}^{(1)}(\tau^j)(E_i) = 
    \tilde{\rho}(\tau^j)^t E_i \tilde{\rho}(\tau^j).
\text{ for all } 1\leq j \leq q.
\]
If we can show for every $E \in \tilde{V}$ we have  $\tilde{\rho}^c(\sigma^j)(E) \in \tilde{V}$, then the curve can be lifted together with the $G$ action. 
 
One might try to deform the matrices $E_1,\ldots, E_r$ to matrices $\tilde{E}_1, \ldots , \tilde{E}_r$ so that $\tilde{E}_i= e_i \mod \mathfrak{m}_R$ so that the free $R$-module $\langle \tilde{E}_1, \ldots , \tilde{E}_r \rangle$, equipped with the $\tilde{\rho}^{(1)}$ action is $G$-invariant. It seems that this is not always possible. For instance, we can take as 
\[
  V=V_{\alpha }(1,2) \subset V_{\alpha }(1,2) \bigoplus
  V_{\alpha }(3,2)=W.
\]
as in the example \cite[p. 777]{MR4779377}. The $G$-module $W$ lifts in characteristic zero by theorem \ref{th:main-lift}, while there is no way to modify the basis of $V$ in order to obtain a $G$-module $\tilde{V}$ of rank $2$ as the original module. 

Writing a sufficient condition on whether the module $\langle e_1, \ldots , e_r \rangle$ can be lifted requires the knowledge of $G$-module structure of 
$\langle e_1, \ldots , e_r\rangle =\mathrm{Tor}^1(k,I_X)$,  see eq.(3) in \cite{2101.11084}. This $G$-module structure is still unkwown. 
}

\section{Galois module structure of holomorphic differentials, special fibre}
\label{sec:HolomorphicDifferentialSpecialFibre}



Consider the group $C_q \rtimes C_m$. Let $\tau$ be a generator of $C_q$ and $\sigma$ a generator of $C_m$.  It is known that $\Aut(C_q)\cong \mathbb{F}_p^* \times Q$, for some abelian group $Q$. The representation $\psi:C_m \rightarrow \Aut(C_q)$ given by the action of $C_m$ on $C_q$ is known to factor through a character $\chi:C_m \rightarrow \mathbb{F}_p^*$. The order of $\chi$
divides $p-1$ and $\chi^{p-1}=\chi^{-(p-1)}$ is the trivial 
{one-dimensional
}
character. 
In our setting, using the definition of $G$ given in eq.\;(\ref{eq:G-def}) and lemma \ref{loga} we have that the character $\chi$ is defined by 
\begin{equation}
\label{eq:chi-def}
  \chi(\sigma)=\alpha=\zeta_m^{a_0} \in \mathbb{F}_p.
\end{equation}
For all $i\in \Z$, $\chi^i$ defines a simple $k[C_m]$-module of $k$ dimension one, which we will denote by $S_{\chi^i}$. For $0 \leq \ell \leq m-1$ denote by $S_\ell$ the simple module on which $\sigma$ acts as $\zeta_m^\ell$.  Both $S_{\chi^i}$, $S_\ell$ can be seen as $k[C_q\rtimes C_m]$-modules using inflation. Finally, for $0\leq \ell \leq m-1$ we define $\chi^i(\ell) \in \{0,1,\ldots,m-1\}$ such that $S_{\chi^i(\ell)} \cong S_\ell \otimes_k S_{\chi^i}$. Using eq. (\ref{eq:chi-def}) we arrive at 
\begin{equation}
\label{eq:chi-prod}
S_{\chi^i(\ell)}=S_{\ell+i a_0}.
\end{equation}

There are $q\cdot m$ isomorphism classes of indecomposable $k[C_q \rtimes C_m]$-modules and are all uniserial, i.e. the set of  submodules are totally ordered by inclusion. An indecomposable $k[C_q \rtimes C_m]$-module $U$ is uniquely determined by its socle, which is the kernel of the action of $\tau-1$ on $U$, and its $k$-dimension. For $0 \leq \ell \leq m-1$ and $1\leq \mu \leq q$, let $U_{\ell,\mu}$ be the indecomposable $k[C_q\rtimes C_m]$-module with socle $S_\ell$ and $k$-dimension $\mu$. Then $U_{\ell,\mu}$ is uniserial, {\cite[rem. 3.4]{MR4130074}} and its $\mu$
ascending composition factors are the first $\mu$ composition factors of the sequence 
\[
S_\ell, S_{\chi^{-1}(\ell)},S_{\chi^{-2}(\ell)},\ldots,S_{\chi^{-(p-2)}(\ell)}, 
S_{\ell}, S_{\chi^{-1}(\ell)},S_{\chi^{-2}(\ell)},\ldots, S_{\chi^{-(p-2)}(\ell)}.
\]
\begin{lemma}
\label{comparisonIndec}
There is the following relation between indecomposable modules:
\[
V_{\alpha}(\lambda,\kappa)=U_{(\lambda+a_0(\kappa-1) \mod m,\kappa)}
\]
In particular, for the case of dihedral groups $D_q$
we have the relation, {$a_0=1$},
\[
  V_{\alpha}(\lambda,\kappa)=U_{(\lambda+\kappa-1  \mod  2,\kappa)}.
\]
\end{lemma}
\begin{proof}
Indeed, in the $V_{\alpha}(\lambda,\kappa)$ notation we describe the action of $\sigma$ on the generator $e$, by assuming that  $\sigma e= \zeta_m^\lambda e$.
We can then describe the action on every basis element $e_i=(\tau-1)^{i-1}e$, using the group relations 
\[
  \sigma e_i =\sigma(\tau-1)^{i-1}e= (\tau^\alpha-1)^{i-1}\sigma e =
  \zeta_m^\lambda (\tau^\alpha-1)^{i-1} e
\]
This allows us to prove, see \cite[lemma 10]{MR4779377} that 
\[
  \sigma e_i =\alpha^{i-1} \zeta_m^\lambda 
  {e_i}
  + \sum_{\nu=i+1}^\kappa a_\nu e_\nu
\]
for some elements $a_\nu\in k$ and in particular
\[
  \sigma e_\kappa =  \alpha^{\kappa-1} \zeta_m^{\lambda}
  {e_\kappa }. 
\]
Recall that the number $\alpha=\zeta_m^{a_0}$ for some natural number $a_0$, $0\leq a_0 < m-1$, see also \cite[lemma 2]{MR4779377}. In the $U_{\mu,\kappa}$ notation, $\mu$ is the action on the one-dimensional socle which is the $\tau$-invariant element 
$e_\kappa=(\tau-1)^{\kappa-1}e$, i.e. $\sigma(e_\kappa)=\zeta_m^\mu$. Putting all this together we have 
\[
  \mu = \lambda + (\kappa-1)a_0 \mod m. 
\]
In the case of dihedral group $D_q$, $m=2$ and $\alpha=-1^{a_0}$, i.e.
$a_0=1$, we have $V_{\alpha}(\lambda,\kappa)=U_{\lambda+\kappa-1 \mod 2,\kappa}$. 
\end{proof}

Assume that $X\rightarrow \mathbb{P}^1$ is an HKG-cover with Galois group $C_q \rtimes C_m$. The  subgroup $I$ generated by the Sylow $p$-subgroups of the inertia groups of all closed points of $X$ is equal to $C_q$, {and the notation of section 4 in \cite{MR4130074} is simplified.} 
\begin{definition}
\label{define-Dj}
In \cite{MR4130074} for 
each $0\leq j \leq q-1$ the divisor 
\[
D_j=\sum_{y\in \mathbb{P}^1} d_{y,j} y,
\]
is defined,
where the integers $d_{y,j}$ are given as follows.
Let $x$ be a point of $X$ above  $y$ and consider the $i$-th ramification group $I_{x,i}$ at $x$. 
The order of the inertia group at $x$ is assumed to be $p^{n(x)}$ and  $i(x)=h-n(x)$ is defined. In this article we will have HKG-covers, where 
$n(x)=h$, so $i(x)=0$. We will use this in order to simplify the notation in what follows. 

Let $b_0,b_1,\ldots,b_{h-1}$ be the jumps in the numbering of the lower ramification filtration subgroups of $I_x$.
 We define
\[
d_{y,j}=
\lf
\frac{1}{p^{h}} \sum_{l=1}^{h}
p^{h-l} 
\big(
p-1+ (p-1-a_{l,t})b_{l-1}
\big)
\rf 
\]  
for all $j\geq 0$ 
with $p$-adic expansion 
\[
j=a_{1,j}+a_{2,j}p+ \cdots+ a_{h,j}p^{h-1}.
\]
 In particular $D_{q-1}=0$. Observe that $d_{y,j}\neq 0$ only for  wildly ramified branch points.   

\end{definition}


\begin{remark}
\label{remark-D}
For a divisor $D$ on a curve $Y$ define $\Omega_Y(D)=\Omega_Y\otimes \O_Y(D)$. In particular for  $Y=\mathbb{P}^1$, and for $D=D_j= d_{P_\infty,j} P_\infty$, where $D_j$ is a divisor supported at the infinity point $P_\infty$
 we have
\begin{align*}
H^0(\mathbb{P}^1,\Omega_{\mathbb{P}^1}(D_j))
&=\{ f(x) dx: 0 \leq \deg f(x) \leq d_{P_\infty,j}-2 \}.
\end{align*}
For the sake of simplicity, we will denote $d_{P_\infty,j}$ by $d_j$. The space $H^0(\mathbb{P}^1,\Omega_{\mathbb{P}^1}(D_j))$
 has a basis given by $B=\{dx,xdx,\ldots,x^{d_j-2}dx\}$. Therefore, 
the number $n_{j,\ell}$ of simple modules appearing in the decomposition $\Omega_{\mathbb{P}^1}(D_j)$ isomorphic to $S_\ell$ for $0\leq \ell < m$, is equal to the number of monomials $x^{\nu}$ with 
\[
\nu \equiv \ell-1 \mod m, 0\leq \nu \leq d_j-2.
\] 
If $d_j \leq 1$ then $B=\emptyset$ and $n_{j,\ell}=0$ for all $0\leq \ell <m$. If $d_j>1$, then we know that in the $d_j-1$ elements of the basis $B$, the first $m \lf \frac{d_j-1}{m} \rf$ elements  contribute to every representative modulo $m$. Thus, we have  at least $\lf \frac{d_j-1}{m} \rf$ elements in isomorphic to $S_\ell$ for every $0\leq \ell <m$. We will now count the rest elements, of the form $\{x^\nu dx\}$, where 
\[
  m \lf \frac{d_j-1}{m} \rf \leq \nu \leq d_j-2 \text{ and } \nu \equiv \overline{\ell-1} \mod m,
\]  
where  $\overline{\ell-1}$ is  the unique integer in $\{0,1,\ldots,m-1\}$ equivalent to $\ell-1$ modulo $m$. 
We observe that the number $y_j(\ell)$ of such elements $\nu$ is given by 
\[
  y_j(\ell) 
=
\begin{cases}
1 &  \text{ if } \overline{\ell-1} \leq d_j-2- m \lf \frac{d_j-1}{m} \rf
\\
0 & \text{ otherwise}
\end{cases}
\]
Therefore 
\[
n_{j,\ell}=
\begin{cases}
\lf \frac{d_j-1}{m}\rf +y_j(\ell) & \text{ if } d_j \geq 2
\\
0 & \text{ if } d_j \leq 1
\end{cases}
\]
For example if $d_j=9$ and $m=3$, then a basis for $H^0(\mathbb{P}^1,\Omega_{\mathbb{P}^1}(9P_\infty))$ is given by $\{dx,x dx,x^2dx,\ldots x^7dx\}$. This basis has $8$ elements, and each triple $\{dx,x dx,x^2 dx\}$, $\{x^3 dx,x^4 dx,x^5 dx\}$ contributes one to each class $S_0,S_1,S_2$, while there are two remaining basis elements $\{x^6dx,x^7dx\}$, which contribute one to $S_1,S_2$. Notice that $\lf \frac{8}{3} \rf=2$ and $y(\ell)=1$ for $\ell=1,2$.

In particular if $m=2$, then $n_{j,\ell}=0$ if $d_j\leq 1$ and for $d_j \geq 2$ we have
\begin{equation}
\label{eq-njl}
n_{j,\ell}=
\begin{cases}
\frac{d_j-1}{2} & \text{ if } d_j \equiv 1 \mod 2
\\
\frac{d_j}{2}-1 & \text{ if } \ell=0 \text{ and } d_j \equiv 0 \mod 2
\\
\frac{d_j}{2} & \text{ if } \ell=1 \text{ and } d_j \equiv 0 \mod 2
\end{cases}
\end{equation}
\end{remark}
\begin{lemma}
\label{lemma:njl}
Let $m=2$ and assume that $d_{j-1}=d_j+1$. Then if $d_{j} \geq 2$
\[
  n_{j-1,\ell}-n_{j,\ell}=
  \begin{cases}
  1  & \text{ if } d_{j-1}\equiv 1 \mod2 \text{ and }  \ell=0
  \\
    &  \text{ or } d_{j-1}\equiv 0 \mod 2 \text{ and } \ell=1
  \\
  0  & \text{ if } d_{j-1}\equiv 1 \mod2 \text{ and }  \ell=1
  \\
    &  \text{ or } d_{j-1}\equiv 0 \mod 2 \text{ and } \ell= 0
  \end{cases}
\]
If $d_{j} \leq 1$, then
\[
  n_{j-1,\ell}-n_{j,\ell}=
  \begin{cases}
  0 & \text{ if } d_{j}=0 \text{ or } (d_j=1 \text{ and } \ell=0) \\
  1 & \text{ if } d_j=1 \text{ and } \ell=1
  \end{cases}
\] 
\end{lemma}
\begin{proof}
Assume that $d_j\geq 2$. 
We distinguish the following two cases, and we use eq. (\ref{eq-njl}) in each case
\begin{itemize}[leftmargin=10pt]
\item $d_{j-1}$ is odd and $d_{j}$ is even. Then, if $\ell=0$
\[
  n_{j-1,\ell}-n_{j,\ell}=\frac{d_{j-1}-1}{2}- \frac{d_{j}}{2}+1=1
\]
while $n_{j-1,\ell}-n_{j,\ell}=0$ if $\ell=1$. 
\item $d_{j-1}$ is even and $d_{j}$ is odd. Then, if $\ell=0$
\[
  n_{j-1,\ell}-n_{j,\ell}=\frac{d_{j-1}}{2}-1- \frac{d_{j}-1}{2}=0,
\]
while $n_{j-1,\ell}-n_{j,\ell}=1$ if $\ell=0$.
\end{itemize}
If now $d_j=0$ and $d_{j-1}=1$, then $n_{j-1,\ell}-n_{j,\ell}=0$. If $d_j=1$ and $d_{j-1}=2$ then $n_{j,\ell}=0$ while $n_{j-1,\ell}=0$ if $\ell=0$ and $n_{j-1,\ell}=1$ if $\ell=1$.
\end{proof}


\begin{theorem}
    \label{th:21}
Let $M=H^0(X,\Omega_X)$, and let $\tau$ be the generator of $C_q$. For all $0\leq j < q$ we define
$M^{(j)}$ to be the kernel of the action 
{of $(\tau-1)^j$.} 
For $0\leq a \leq m-1$ and $1\leq b \leq q=p^h$, let $n(a,b)$ be the number of indecomposable 
direct $k[C_q \rtimes C_m]$-module summands of $M$ that are isomorphic to $U_{a,b}$. 
Let $n_1(a,b)$ be the number of indecomposable direct $k[C_m]$-summands of $M^{(b)}/M^{(b-1)}$ with socle $S_{\chi^{-(b-1)}(a)}$ and dimension $1$. Let $n_2(a,b)$ be the number of indecomposable
direct $k[C_m]$-module summands of $M^{(b+1)}/M^{(b)}$ with socle $S_{\chi^{-b}(a)}$, where 
we set $n_2(a,b)=0$ if $b=q$. 
{hen, 
\begin{itemize}
\item
$n(a,b)=n_1(a,b)-n_2(a,b)$. 
\item The numbers $n_1(a,b),n_2(a,b)$ can be computed using the isomorphism 
\[
M^{(j+1)}/M^{(j)}\cong S_{\chi^{-j}} \otimes_k H^0(Y,\Omega_Y(D_j)), 
\]
\end{itemize} 
}
where $Y=X/C_q$ and $D_j$ are the divisors on $Y$, given in definition \ref{define-Dj}.

\end{theorem}
\begin{proof}
This theorem is proved in \cite{MR4130074}, see remark 4.4.
\end{proof}


\begin{corollary}
\label{cor:6}
Using the notation of theorem \ref{th:21}, we set
\begin{equation}
    \label{eq:defdj}
d_j=\lf
\frac{1}{p^{h}} \sum_{l=1}^h
p^{h-l} 
(p-1+ (p-1-a_{l,t})b_{l-1})
\rf.
\end{equation}
The numbers $n(a,b)$, $n_1(a,b)$ and $n_2(a,b)$ defined in theorem \ref{th:21} are given by
\[
  n(a,b)=n_1(a,b)-n_2(a,b)=n_{b-1,a} -n_{b,a}.
\]
\end{corollary}
\begin{proof}
We treat the $n_1(a,b)$ case and the $n_2(a,b)$ follows similarly.
By {theorem \ref{th:21}} 
we have that 
\[
  M^{(b)}/M^{(b-1)} \cong S_{\chi^{-(b-1)}} \otimes_k H^0(\mathbb{P}^1,\Omega_{\mathbb{P}^1}(D_b)).
\]
The number of indecomposable $k[C_m]$-summands of $M^{(b)}/M^{(b-1)}$ isomorphic to 
$S_{ {\chi^{-(b-1)}(a)}}=S_{a -(b-1)a_0}$ equals to the number of indecomposable $k[C_m]$-summands of $H^0(\mathbb{P}^1,\Omega_{\mathbb{P}^1}(D_{b}))$ isomorphic to $S_a$, which is computed in remark \ref{remark-D}.
\end{proof}
In \cite[Th. 1.1]{MR2577662} A. Obus and R. Pries described the upper jumps in the ramification filtration of $C_{p^h} \rtimes C_m$-covers. 
\begin{theorem}
Let $G=C_{p^h} \rtimes C_m$, where $p\nmid m$. Let $m'= |\mathrm{Cent}_G(\sigma)|/p^h$, where $\langle \tau \rangle= C_{p^h}$. A sequence $u_1 \leq \cdots \leq u_n$ of rational numbers occurs as the set of positive breaks in the upper numbering of the ramification filtration of a $G$-Galois extension of $k( (t))$ if and only if:
\begin{enumerate}
\item $u_i \in \frac{1}{m}\mathbb{N}$ for $1\leq i \leq h$.
\item $\mathrm{gcd}(m,mu_1)=m'$.
\item $p\nmid mu_1$ and for $1 < i \leq h$, either $u_i=p u_{i-1}$ or both $u_i > p u_{i-1}$ and $p\nmid m u_i$.
\item $m u_i \equiv m u_1 \mod m$ for $1 \leq i \leq n$.
\end{enumerate}
\end{theorem}
{We will now describe the lower  $b_0, \ldots , b_{h-1}$ and upper $w_0, \ldots , w_{h-1}$ for the case of a cyclic group action.  
}
Notice that in our setting $\mathrm{Cent}_G(\tau)=\langle \tau \rangle$, therefore $m'=1$. Also the set of upper jumps of $C_{p^h}$ is given by $w_1= mu_1,\ldots, w_h=m u_h, w_i \in \mathbb{N}$, see \cite[lemma 3.5]{MR2577662}. 

The theorem of Hasse-Arf \cite[p. 77]{SeL} applied for cyclic groups, implies that there are strictly positive integers $\iota_0,\iota_1,\ldots,\iota_{h-1}$ such that  
\[
b_s=\sum_{\nu=0}^{s-1} \iota_\nu p^\nu, \text{ for } 0 \leq s \leq h-1.
\]
 Also, 
the upper jumps for the $C_q$ extension are given by 
\begin{equation}
\label{eq-w-upper}
  w_0=i_0-1, w_1=i_0+i_1-1,\ldots,w_h=i_0+i_1+\cdots+u_h-1.
\end{equation}
Assume that for all $0 < \nu \leq h-1$ we have $w_{\nu}= p w_{\nu-1}$. 
Equation (\ref{eq-w-upper}) implies that 
\[
i_1=(p-1)w_0, i_2=(p-1)pw_0,i_3=(p-1)p^2w_0,\ldots, u_{h-1}=(p-1)p^{h-2}w_0.
\]
Therefore, 
\begin{align*}
  b_\ell+1 &=\sum_{\nu=0}^{\ell} i_\nu p^\nu
  \\
  &=1 +w_0 +(p-1)w_0 \cdot p + (p-1)p w_0 \cdot p^2
  \cdots+ (p-1)p^{\ell-1}w_0 \cdot p^\ell
\\
&=
1+u_0 + p(p-1)u_0 
\left(
\sum_{\nu=0}^{\ell-1} 
p^{2 \nu}
\right)
=
1+ w_0 +  p(p-1)w_0 \frac{p^{2 \ell}-1}{p^2-1}
\\
&=
1+w_0+ pw_0 \frac{p^{2 \ell}-1}{p+1}
=
1+ w_0\frac{p^{2 \ell+1}+1}{p+1},
\end{align*}
where we have used that $w_0=b_0=i_0-1$.

\subsection{Examples}
\label{sec:no-lift}

Consider the curve with lower jumps $1,21,521$ and higher jumps $1,5,25$, acted on by $C_{125} \rtimes C_4$. According to eq. (\ref{a-modq}), the only possible values for $\alpha$ are $1,57,68,124$. The value $\alpha=1$ gives rise to a cyclic group $G$, while the value $\alpha=124$ has order $2$ modulo $125$. The values $57,68$ have order $4$ modulo $125$. The cyclic group $\mathbb{F}_5^*$ is generated by the primitive root $2$ of order $4$. We have that $57\equiv 2 \mod5$, while $68\equiv 3 \equiv 2^3 \mod5$. 
We have thus two choices for the $\zeta_4$, namely $2\mod 5$ as well as $\zeta_4^3= 2^3=3 \mod 5$. In remark \ref{remark-D} we have considered the $m$-th root of unity such that $\sigma(x) = \zeta_m x$, where $x$ is the generating variable for the function field of the curve $X^{C_q}$. 
On the other hand the curve 
$X^{\langle \tau^{p^{h-1}}\rangle}$ 
is a $C_p \rtimes C_m$ cover of $\mathbb{P}^1$ with cyclic Galois group generated by $\tilde{\tau }=\tau \mod \langle \tau^{p^{h-1}} \rangle$ 
and 
\[
  \sigma \tilde{\tau } \sigma^{-1} = \tilde{\tau}^\alpha 
\]
According to \cite[lemma 1.4.1]{Pries:02} it has the following model:
\[
  x^m=u, y^p-y=f(x).
\]
and if the action of $\sigma$ on $x$ is given by 
$\sigma(x)=\zeta_m x$, then $\alpha=\zeta_m^{-j} \mod m$, where $j=\deg(f)$ and equals to the first upper jump $u_0$ for $C_p$. 

 Assume that $\alpha=\zeta_m^{-1}$, that is $u_0\equiv 1 \mod m$. We will now use theorem \ref{cond-lift} in order to show that this  action does not lift.
Using corollary \ref{cor:6} together with remark \ref{remark-D} we have that $H^0(X,\Omega_X)$ is decomposed into the following indecomposable modules, each one appearing with multiplicity one:

$
\begin{array}{c}
U_{ 0 , 5 },\ 
U_{ 3 , 11 },\ 
U_{ 2 , 17 },\ 
U_{ 1 , 23 }, \
U_{ 0 , 29 },\ 
U_{ 3 , 35 },\ 
U_{ 2 , 41 },\ 
U_{ 1 , 47 },\ 
U_{ 0 , 53 },\ 
U_{ 3 , 59 },\ 
\\
U_{ 2 , 65 },\ 
U_{ 1 , 71 },\ 
U_{ 0 , 77 },\ 
U_{ 3 , 83 },\ 
U_{ 2 , 89 },\ 
U_{ 1 , 95 },\ 
U_{ 0 , 101 },\ 
U_{ 3 , 107 },\ 
U_{ 2 , 113 },\ 
U_{ 1 , 119 }.
\end{array} 
$
\\
{
For all $U_{l,\kappa}$ listed above we have 
$ l \not \equiv 0\mod 4$ so the module $U_{l,\kappa}$ can not be lifted by itself. 
We will now examine possible matchings of modules. Recall that we are looking for sets  
$U_{l_1,\kappa_1}, U_{l_1,\kappa_2},\ldots , 
U_{l_t,\kappa_t}$ so that $\kappa_1+ \cdots +\kappa_t \leq p^h$.
Trying to satisfy only the dimension criterion we see that  only possible matchings are $\{U_{1,119},U_{0,5}\}$,
$\{U_{2,113},U_{3,11}\}$, $\{U_{3,107}, U_{2,17}\}$ etc.
We thus have the following vertical matchings for the modules:

\centerline{
\begin{tabular}{|c|c|c|c|c|c|c|c|c|c|}
        \hline
  $U_{ 0 , 5 }  $ &
  $U_{ 3 , 11 } $ &
  $U_{ 2 , 17 } $ &
  $U_{ 1 , 23 } $ &
  $U_{ 0 , 29 } $ &
  $U_{ 3 , 35 } $ &
  $U_{ 2 , 41 } $ &
  $U_{ 1 , 47 } $ &
  $U_{ 0 , 53 } $ &
  $U_{ 3 , 59 } $     
        \\
        \hline
        $U_{ 1 , 119 } $ & 
        $U_{ 2 , 113 } $ &
$U_{ 3 , 107 } $ &
$U_{ 0 , 101 } $ &
$U_{ 1 , 95 } $  &
$U_{ 2 , 89 } $  &
$U_{ 3 , 83 } $  &
$U_{ 0 , 77 } $  &
$U_{ 1 , 71 } $  &
 $U_{ 2 , 65 } $  
        \\
        \hline
\end{tabular}
}

Observe that the sum of all dimensions vertically are 
$124=5^3-1$.
The pair 
$\{U_{0,5}, U_{1,119}\}$
does not satisfy criterion \ref{cond-liftc} of theorem \ref{cond-lift}. Indeed, for $l_1=0, \kappa _1=5$ and $l_2=1, \kappa_2=119$ we have $\alpha=\zeta_4^{-1}$ and 
$l_1+\kappa_1 (-1) = 3 \neq l_2=1 \mod 4 $ and also 
$l_2 + \kappa_2(-1)=2 \neq l_1=0 \mod 4$.
Therefore, the action cannot be lifted. 


\begin{remark}
Note that the only possible choice for $\alpha$ other than $\zeta_4^{-1}$ is $\alpha=\zeta_4$. According to \cite[lemma 1.4.1]{Pries:02}, this forces the first lower jump to be congruent to $-1$ modulo $4$. Equivalently, as noted in remark \ref{remark:HKB}, the KGB obstruction vanishes. At the same time, the criterion \ref{cond-liftc} of theorem \ref{cond-lift} is satisfied for all pairs, and therefore the representation can be lifted.
\end{remark}

The above  example has non-vanishing KGB obstruction, see remark \ref{remark:HKB}, so our criterion does not give something new here.
The case of dihedral groups, in which the KGB-obstruction is always vanishing, is more difficult to find an example that does not lift. 

Let us now consider the case of dihedral groups $D_{125}=C_{125} \rtimes C_2$, that is $m=2$ and assume that the  lower jumps are $1,21,521$ and higher jumps are $1,5,25$. The set of indecomposable modules  is given by the following table:

\centerline{
\begin{tabular}{|c|c|c|c|c|c|c|c|c|c|}
        \hline
  $U_{ 0 , 5 }  $ &
  $U_{ 1 , 11 } $ &
  $U_{ 0 , 17 } $ &
  $U_{ 1 , 23 } $ &
  $U_{ 0 , 29 } $ &
  $U_{ 1 , 35 } $ &
  $U_{ 0 , 41 } $ &
  $U_{ 1 , 47 } $ &
  $U_{ 0 , 53 } $ &
  $U_{ 1 , 59 } $     
        \\
        \hline
        $U_{ 1 , 119 } $ & 
        $U_{ 0 , 113 } $ &
$U_{ 1 , 107 } $ &
$U_{ 0 , 101 } $ &
$U_{ 1 , 95 } $  &
$U_{ 0 , 89 } $  &
$U_{ 1 , 83 } $  &
$U_{ 0 , 77 } $  &
$U_{ 1 , 71 } $  &
 $U_{ 0 , 65 } $  
        \\
        \hline
\end{tabular}
}
which is exactly the set of indecomposable groups $U_{l, \kappa}$ for $C_{125} \rtimes C_4$ but $l$ is reduced modulo $2$. Now $\alpha =-1$, that is $a_0=1$ and for all vertical 
pairs of  modules $U_{l_1,\kappa_1}, U_{l_2,\kappa_2}$ we have that $l_1+ \kappa_1=l_2$, that is criterion \ref{cond-liftc} of theorem \ref{cond-lift} is satisfied.
This indicates that, in this liftable case, our criterion is consistent with the vanishing of the KGB-obstruction.

Let us now give an example of dihedral group which does not lift. 
The HKG-cover with lower jumps $9, 9\cdot21=189,9\cdot 521=4689$ has genus $11656$ and the following modules appear in its decomposition, each one appearing with multiplicity one:
\[
\tiny{
\begin{array}{l}
U_{ 0 , 1 },
U_{ 1 , 1 },
U_{ 0 , 2 },
U_{ 1 , 2 },
U_{ 1 , 3 },
U_{ 0 , 4 },
U_{ 1 , 4 },
U_{ 0 , 5 },
U_{ 1 , 6 },
U_{ 0 , 7 },
U_{ 1 , 7 },
U_{ 0 , 8 },
U_{ 1 , 8 },
U_{ 0 , 9 },
U_{ 1 , 9 },
U_{ 0 , 11 },
U_{ 1 , 11 },
U_{ 0 , 12 },
\\
U_{ 1 , 12 },
U_{ 0 , 13 },
U_{ 1 , 13 },
U_{ 0 , 14 },
U_{ 1 , 15 },
U_{ 0 , 16 },
U_{ 0 , 17 },
U_{ 1 , 17 },
U_{ 0 , 18 },
U_{ 1 , 18 },
U_{ 0 , 19 },
U_{ 1 , 19 },
U_{ 0 , 21 },
U_{ 1 , 21 },
U_{ 0 , 22 },
U_{ 1 , 22 },
\\
U_{ 0 , 23 },
U_{ 1 , 23 },
U_{ 1 , 24 },
U_{ 0 , 25 },
U_{ 1 , 26 },
U_{ 0 , 27 },
U_{ 1 , 27 },
U_{ 0 , 28 },
U_{ 1 , 28 },
U_{ 0 , 29 },
U_{ 1 , 29 },
U_{ 0 , 31 },
U_{ 1 , 31 },
U_{ 0 , 32 },
U_{ 1 , 32 },
U_{ 0 , 33 },
\\
U_{ 0 , 34 },
U_{ 1 , 34 },
U_{ 1 , 35 },
U_{ 0 , 36 },
U_{ 0 , 37 },
U_{ 1 , 37 },
U_{ 0 , 38 },
U_{ 1 , 38 },
U_{ 0 , 39 },
U_{ 1 , 39 },
U_{ 0 , 41 },
U_{ 1 , 41 },
U_{ 0 , 42 },
U_{ 1 , 42 },
U_{ 0 , 43 },
U_{ 1 , 43 },
\\
U_{ 1 , 44 },
U_{ 0 , 45 },
U_{ 0 , 46 },
U_{ 1 , 46 },
U_{ 1 , 47 },
U_{ 0 , 48 },
U_{ 1 , 48 },
U_{ 0 , 49 },
U_{ 1 , 49 },
U_{ 0 , 51 },
U_{ 1 , 51 },
U_{ 0 , 52 },
U_{ 1 , 52 },
U_{ 0 , 53 },
U_{ 0 , 54 },
U_{ 1 , 54 },
\\
U_{ 1 , 55 },
U_{ 0 , 56 },
U_{ 0 , 57 },
U_{ 1 , 57 },
U_{ 0 , 58 },
U_{ 1 , 58 },
U_{ 0 , 59 },
U_{ 1 , 59 },
U_{ 0 , 61 },
U_{ 1 , 61 },
U_{ 0 , 62 },
U_{ 1 , 62 },
U_{ 0 , 63 },
U_{ 1 , 63 },
U_{ 1 , 64 },
U_{ 0 , 65 },
\\
U_{ 0 , 66 },
U_{ 1 , 66 },
U_{ 1 , 67 },
U_{ 0 , 68 },
U_{ 1 , 68 },
U_{ 0 , 69 },
U_{ 1 , 69 },
U_{ 0 , 71 },
U_{ 1 , 71 },
U_{ 0 , 72 },
U_{ 1 , 72 },
U_{ 0 , 73 },
U_{ 1 , 73 },
U_{ 0 , 74 },
U_{ 1 , 75 },
U_{ 0 , 76 },
\\
U_{ 0 , 77 },
U_{ 1 , 77 },
U_{ 0 , 78 },
U_{ 1 , 78 },
U_{ 0 , 79 },
U_{ 1 , 79 },
U_{ 0 , 81 },
U_{ 1 , 81 },
U_{ 0 , 82 },
U_{ 1 , 82 },
U_{ 0 , 83 },
U_{ 1 , 83 },
U_{ 1 , 84 },
U_{ 0 , 85 },
U_{ 1 , 86 },
U_{ 0 , 87 },
\\
U_{ 1 , 87 },
U_{ 0 , 88 },
U_{ 1 , 88 },
U_{ 0 , 89 },
U_{ 1 , 89 },
U_{ 0 , 91 },
U_{ 1 , 91 },
U_{ 0 , 92 },
U_{ 1 , 92 },
U_{ 0 , 93 },
U_{ 1 , 93 },
U_{ 0 , 94 },
U_{ 1 , 95 },
U_{ 0 , 96 },
U_{ 1 , 96 },
U_{ 0 , 97 },
\\
U_{ 0 , 98 },
U_{ 1 , 98 },
U_{ 0 , 99 },
U_{ 1 , 99 },
U_{ 0 , 101 },
U_{ 1 , 101 },
U_{ 0 , 102 },
U_{ 1 , 102 },
U_{ 1 , 103 },
U_{ 0 , 104 },
U_{ 1 , 104 },
U_{ 0 , 105 },
U_{ 1 , 106 },
U_{ 0 , 107 },
\\
U_{ 1 , 107 },
U_{ 0 , 108 },
U_{ 1 , 108 },
U_{ 0 , 109 },
U_{ 1 , 109 },
U_{ 0 , 111 },
U_{ 1 , 111 },
U_{ 0 , 112 },
U_{ 1 , 112 },
U_{ 0 , 113 },
U_{ 1 , 113 },
U_{ 0 , 114 },
U_{ 1 , 115 },
U_{ 0 , 116 },
\\
U_{ 1 , 116 },
U_{ 0 , 117 },
U_{ 0 , 118 },
U_{ 1 , 118 },
U_{ 0 , 119 },
U_{ 1 , 119 },
U_{ 0 , 121 },
U_{ 1 , 121 },
U_{ 0 , 122 },
U_{ 1 , 122 },
U_{ 0 , 123 },
U_{ 1 , 123 },
U_{ 1 , 124 }.
\end{array}
}
\]
The above formulas were computed using Sage 9.8 \cite{sagemath9.8}. In order to be completely sure that the computations are correct we will compute the values we need by hand also. 
We have
\begin{align*}
  d_j &=\lf \frac{1}{125}\left( 5^2
  \big(4 +(4-a_1) 9  \big)
  +
  5\big(4 +(4-a_2) 189 \big)
  + 
  \big(4+(4-a_3)4689 \big) \right) \rf
  \\
  &= \lf
\frac{1}{125} \left( 
23560 - 225 a_1 - 945 a_2 - 4689 a_3
\right)
   \rf
\end{align*}
\[
\def\arraystretch{1.15}
\begin{array}{c|c|c|c|c|c|c}
j & p\!-\!\text{adic} & d_j & n_{j,0} & n_{j,1} & n_{j-1,0}-n_{j,0} & n_{j-1,1}-n_{j,1}
\\
\hline
0 & 0,0,0& \lf \frac{23560}{125} \rf=188 & 93 & 94 & - & -
\\
1 & 1,0,0& \lf \frac{23335}{125} \rf=186 & 92 & 93 & 1 & 1
\\
2 & {2},0,0& \lf \frac{23110}{125} \rf=184 & 91 & 92 & 1 & 1
\\
3 & {3},0,0& \lf \frac{22885}{125} \rf=183 &  91 & 91 & 0 & 1
\\
4 & {4},0,0& \lf \frac{22660}{125} \rf=181 &  90 & 90 & 1 & 1 
\\
5 & 0,1,0& \lf \frac{22615}{125} \rf=180 & 89 & 90 & 1 & 0
\\ 
6 & 1,1,0 & \lf \frac{22390}{125} \rf=179 & 89 & 89 & 0 & 1
\\
\vdots & \vdots & \vdots & \vdots & \vdots & \vdots & \vdots 
\\
120 & 0,4,4 & \lf \frac{1024}{125} \rf=8 & 3 & 4 &
\\
121 & 1,4,4 & \lf \frac{799}{125} \rf=6 & 2 & 3  & 1 & 1
\\ 
122 & 2,4,4 & \lf \frac{574}{125} \rf=4 & 1 & 2 & 1 & 1
\\ 
123 & 3,4,4 & \lf \frac{349}{125} \rf=2 & 0 & 1 & 1 & 1
\\
124 & 4,4,4 & \lf \frac{124}{125} \rf=0 & 0 & 0 & 0 & 1
\end{array}
\]

Notice that $U_{1,123},U_{0,123}$ can be paired with $U_{1,0},U_{1,1}$, and then for  $U_{0,121}$, $U_{1,121}$ there is only one $U_{1,3}$ to be paired with. The lift is not possible. 


\bibliographystyle{plainurl}

 \def\cprime{$'$}

\end{document}